\documentclass{article}
\usepackage{graphicx} 
\usepackage[utf8]{inputenc}
\usepackage{amsmath}
\usepackage{amsfonts}
\usepackage{amssymb,dsfont}
\usepackage{amsthm}
\usepackage{comment}
\usepackage{diagbox}
\usepackage{hyperref}
\usepackage{ulem}
\usepackage{algpseudocode}
\usepackage{bm}
\usepackage[]{mdframed}
\usepackage{tikz}
\usetikzlibrary{shapes.geometric, arrows}
\tikzstyle{startstop} = [ellipse, minimum width=2cm, minimum height=1cm, text centered, text width=2cm, draw=black]
\tikzstyle{process} = [rectangle, minimum width=2.5cm, minimum height=2cm, text centered, text width=2cm, draw=black]
\tikzstyle{processlong} = [rectangle, minimum width=7cm, minimum height=2cm, text centered, text width=6.5cm, draw=black]
\tikzstyle{decision} = [diamond, minimum width=1cm, minimum height=1cm, text centered, text width=2cm, draw=black]
\tikzstyle{decisionbig} = [diamond, minimum width=1cm, minimum height=1cm, text centered, text width=2.5cm, draw=black]
\tikzstyle{SCC} = [circle, minimum width=0.4cm, text centered, text width=0.4cm, draw=white]
\tikzstyle{arrow} = [thick,->,>=stealth]

\newtheorem{theorem}{Theorem}[section]
\newtheorem{corollary}{Corollary}[theorem]
\newtheorem{lemma}[theorem]{Lemma}
\newtheorem{proposition}[theorem]{Proposition}
\theoremstyle{definition}

\theoremstyle{remark}

\theoremstyle{definition}

\title{On the number of Hamiltonian cycles in the generalized Petersen graph}
\author{Jan Kristian Haugland \\ \texttt{admin@neutreeko.net}}

\begin{document}

\maketitle

\section{Introduction}
Let $n, k$ be positive integers. The generalized Petersen graph $G(n, k)$ is a cubic graph with vertex set $$V(G(n, k)) = \{v_i\}_{0 \leq i < n} \cup \{w_i\}_{0 \leq i < n}$$ and edge set $$E(G(n, k)) = \{v_i v_{i+1}\}_{0 \leq i < n} \cup \{w_i w_{i+k}\}_{0 \leq i < n} \cup \{v_i w_i\}_{0 \leq i < n}$$ where the indices are taken modulo $n$. A Hamiltonian cycle of a graph is a cycle of edges that visits each vertex exactly once. Schwenk \cite{schwenk} found the number of Hamiltonian cycles in $G(n, k)$ for $k=2$, and asked for a corresponding enumeration for higher values of $k$. The objective of this paper is to give initial conditions and linear recurrence relations in the cases $k=3$ and $k=4$, from which explicit formulae can be constructed based on the roots of the characteristic polynomials.

To this end, we introduce the graph $G'(n, k)$ with the same vertex set and edge set as $G(n, k)$, except for the additional vertices $L_0, \dotsc, L_k$, $R_0, \dotsc, R_k$, and the edges $v_0 v_{n-1}$, $w_0 w_{n-k}, \dotsc, w_{k-1} w_{n-1}$ replaced by $v_0 L_0$, $w_0 L_1, \dotsc,$ $w_{k-1} L_k$, $v_{n-1} R_0$, $w_{n-k} R_1, \dotsc, w_{n-1} R_k$. Figure~\ref{gdash} shows this graph in the case $k=3$.
\begin{figure}
\centering
\includegraphics[width=1.0\linewidth]{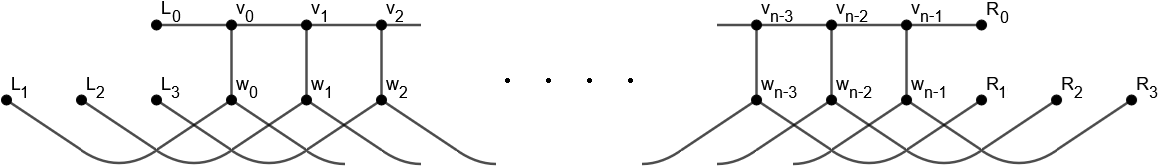}
\caption{The graph $G'(n, 3)$}
\label{gdash}
\end{figure}
The visible edges in the figure are called the \textit{outer} edges. Generally, those that are incident with at least one of $v_0, v_1, \dotsc, v_{k-1}$, $w_0, w_1, \dotsc, w_{k-1}$ are called the outer edges on the left hand side, and those that are incident with at least one of $v_{n-k}, v_{n-k+1}, \dotsc, v_{n-1}$, $w_{n-k}, w_{n-k+1}, \dotsc, w_{n-1}$ are called the outer edges on the right hand side.

Consider a subgraph $\Gamma$ of $G'(n, k)$ with the same vertex set as that of $G'(n, k)$ for which each vertex of degree 3 in $G'(n, k)$ has degree 2 in $\Gamma$, and containing no cycles except possibly one that constitutes the whole of $\Gamma$. Such subgraphs are called \textit{admissible} subgraphs of $G'(n, k)$. Each admissible subgraph $\Gamma$ is associated with a list of pairs of the vertices among $L_0, \dotsc, L_k$, $R_0$, $R_1, \dotsc, R_k$ of degree 1 in $\Gamma$ (called \textit{loose ends} henceforth) that are in the same connected component. This list is called a \textit{list of connections}.

For any given $k$, there is a finite number of possible intersections of an admissible subgraph and the outer edges on either side, which we label $S_1, S_2, \dotsc$. For $k=3$, there are 33 possibilities, as seen in Table~\ref{intersections}. For $k=4$ and $k=5$, the number of possibilities is 85 and 217, respectively.

\begin{table}
\begin{tabular}{|cc|cc|cc|}\hline
& & & & & \\
\raisebox{12pt}{$S_1$} & \includegraphics[width=0.2\linewidth]{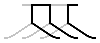} & \raisebox{12pt}{$S_{12}$} & \includegraphics[width=0.2\linewidth]{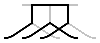} & \raisebox{12pt}{$S_{23}$} & \includegraphics[width=0.2\linewidth]{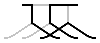} \\
\raisebox{12pt}{$S_2$} & \includegraphics[width=0.2\linewidth]{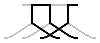} & \raisebox{12pt}{$S_{13}$} & \includegraphics[width=0.2\linewidth]{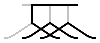} & \raisebox{12pt}{$S_{24}$} & \includegraphics[width=0.2\linewidth]{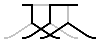} \\
\raisebox{12pt}{$S_3$} & \includegraphics[width=0.2\linewidth]{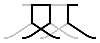} & \raisebox{12pt}{$S_{14}$} & \includegraphics[width=0.2\linewidth]{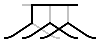} & \raisebox{12pt}{$S_{25}$} & \includegraphics[width=0.2\linewidth]{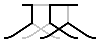} \\
\raisebox{12pt}{$S_4$} & \includegraphics[width=0.2\linewidth]{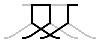} & \raisebox{12pt}{$S_{15}$} & \includegraphics[width=0.2\linewidth]{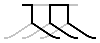} & \raisebox{12pt}{$S_{26}$} & \includegraphics[width=0.2\linewidth]{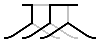} \\
\raisebox{12pt}{$S_5$} & \includegraphics[width=0.2\linewidth]{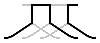} & \raisebox{12pt}{$S_{16}$} & \includegraphics[width=0.2\linewidth]{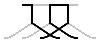} & \raisebox{12pt}{$S_{27}$} & \includegraphics[width=0.2\linewidth]{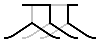} \\
\raisebox{12pt}{$S_6$} & \includegraphics[width=0.2\linewidth]{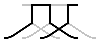} & \raisebox{12pt}{$S_{17}$} & \includegraphics[width=0.2\linewidth]{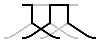} & \raisebox{12pt}{$S_{28}$} & \includegraphics[width=0.2\linewidth]{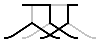} \\
\raisebox{12pt}{$S_7$} & \includegraphics[width=0.2\linewidth]{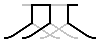} & \raisebox{12pt}{$S_{18}$} & \includegraphics[width=0.2\linewidth]{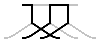} & \raisebox{12pt}{$S_{29}$} & \includegraphics[width=0.2\linewidth]{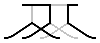} \\
\raisebox{12pt}{$S_8$} & \includegraphics[width=0.2\linewidth]{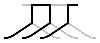} & \raisebox{12pt}{$S_{19}$} & \includegraphics[width=0.2\linewidth]{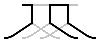} & \raisebox{12pt}{$S_{30}$} & \includegraphics[width=0.2\linewidth]{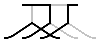} \\
\raisebox{12pt}{$S_9$} & \includegraphics[width=0.2\linewidth]{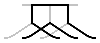} & \raisebox{12pt}{$S_{20}$} & \includegraphics[width=0.2\linewidth]{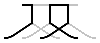} & \raisebox{12pt}{$S_{31}$} & \includegraphics[width=0.2\linewidth]{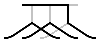} \\
\raisebox{12pt}{$S_{10}$} & \includegraphics[width=0.2\linewidth]{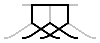} & \raisebox{12pt}{$S_{21}$} & \includegraphics[width=0.2\linewidth]{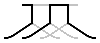} & \raisebox{12pt}{$S_{32}$} & \includegraphics[width=0.2\linewidth]{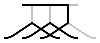} \\
\raisebox{12pt}{$S_{11}$} & \includegraphics[width=0.2\linewidth]{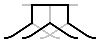} & \raisebox{12pt}{$S_{22}$} & \includegraphics[width=0.2\linewidth]{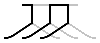} & \raisebox{12pt}{$S_{33}$} & \includegraphics[width=0.2\linewidth]{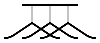} \\ \hline
\end{tabular}
\caption{Possible intersections of an admissible subgraph and the outer edges on either side for $k=3$}
\label{intersections}
\end{table}
Suppose $\Gamma$ is an admissible subgraph of $G'(n, k)$. If the intersection of $\Gamma$ and the outer edges on the left hand side is $S_{n_L}$, the intersection of $\Gamma$ and the outer edges on the right hand side is $S_{n_R}$ and the list of connections is $\lambda$, we say that the $signature$ of $\Gamma$ is $(n_L, n_R, \lambda)$. For $k=3$, there are 1705 distinct signatures in which all the loose ends are accounted for in the list of connections, and two loose ends are in the same pair if they are connected through the outer edges on one side. For example, if $n_L=8$, then necessarily $(L_1, L_2) \in \lambda$. Not all of them can be realized; for example, no admissible subgraph can have signature $(1, 7, \{(R_0, R_3)\})$. For $k = 4$ and $k = 5$, the number of signatures is 25675 and 455835, respectively.

For a fixed $k$, let
\begin{equation}\label{sequence}
u^{(n_L, n_R, \lambda)}_n
\end{equation}
denote the number of admissible subgraphs of $G'(n, k)$ with signature $(n_L, n_R, \lambda)$. We are going to find a linear recurrence relation satisfied by all these sequences for $k=3$ and for $k=4$.

Let $H_k$ denote the set of signatures that pass the following test.

\begin{mdframed}\textbf{Checking the hamiltonicity induced by a signature}

\noindent Start with any loose end. After an even number of steps, move from the current loose end to the one it is connected to according to the list of connections. After an odd number of steps, move from $L_i$ to $R_i$ or from $R_i$ to $L_i$, depending on the current loose end. If this procedure traverses all loose ends, then identifying $L_0$ with $v_{n-1}$, $L_1$ with $w_{n-k}$, $\dotsc$, $L_k$ with $w_{n-1}$, $R_0$ with $v_0$, $R_1$ with $w_0$, $\dotsc$, and $R_k$ with $w_{k-1}$ yields a Hamiltonian cycle in $G(n, k)$.
\end{mdframed}

Thus, the number of Hamiltonian cycles in $G(n, k)$ is given by
\begin{equation}\label{hamiltonian}
h_k(n) = \sum_{(n_L, n_R, \lambda) \in H_k} u^{(n_L, n_R, \lambda)}_n
\end{equation}
As a sum of (\ref{sequence}) over a subset of the possible signatures, $h_k(n)$ also satisfies any linear recurrence relation satisfied by all the individual sequences.

\section{The graph \texorpdfstring{$\sigma_k$}{s} of signatures}
The directed graph $\sigma_k$ is defined as follows. $V(\sigma_k)$ is the set of all signatures. Thus, the terms "vertices" and "signatures" in $\sigma_k$ can be used interchangeably, but we will primarily use the latter. Given a signature $(n_L, n_R, \lambda)$, let $\Gamma$ be a hypothetical (not necessarily realizable) admissible subgraph of $G'(n, k)$ having that signature, for some $n$. The graph $G'(n-1, k)$ can be obtained by removing $v_{n-1} R_0$, $v_{n-1} w_{n-1}$ and $w_{n-1} R_k$ from $E(G'(n, k))$ and $R_0$ and $R_k$ from $V(G'(n, k))$, and renaming $v_{n-1}$ to $R_0$, $R_{k-1}$ to $R_k$, $R_{k-2}$ to $R_{k-1}, \dotsc, R_1$ to $R_2$ and $w_{n-1}$ to $R_1$. If $(n_L, n'_R, \lambda')$ is a possible signature of the intersection of $\Gamma$ and $G'(n - 1, k)$ thus constructed, then the arc $((n_L, n'_R, \lambda'), (n_L, n_R, \lambda))$ is in $E(\sigma_k)$.

For the remainder of this Section, we outline the possible signatures of the intersection of $\Gamma$ and $G'(n-1, k)$.

$v_{n-k-1}w_{n-k-1}$ can not be in $E(\Gamma)$ if either $v_{n-k-1}v_{n-k}, \dotsc, v_{n-2}v_{n-1}$, \newline $v_{n-1}w_{n-1}$ and $w_{n-k-1}w_{n-1}$ are all in $E(\Gamma)$, as a cycle not constituting the whole of $\Gamma$ would be formed, or if $v_{n-k-1}$ and $w_{n-k-1}$ are connected through the outer edges on the right hand side to two loose ends $x$ and $y$ such that $(x, y) \not \in \lambda$.

There are five possibilities for which of the edges $v_{n-1} R_0$, $v_{n-1} w_{n-1}$, $w_{n-1} R_k$ are in $E(\Gamma)$: All three, any two (which covers three possibilities), or only $v_{n-1} w_{n-1}$. Here is how to construct $\lambda'$ from the list of connections $\lambda$ in $\Gamma$ in each case, before the renaming step (recall that $v_{n-1}$ is to be renamed to $R_0$, $R_{k-1}$ to $R_k$ etc.).

\begin{itemize}
\item If all three edges are in $E(\Gamma)$ and $\lambda$ contains other pairs than $(R_0, R_k)$, then remove $(R_0, R_k)$. Note that if $\lambda$ does not contain other pairs than $(R_0, R_k)$, then the edges other than those three form a cycle.
\item If only $v_{n-1} R_0$ and $v_{n-1} w_{n-1}$ are in $E(\Gamma)$, replace $R_0$ by $w_{n-1}$.
\item If only $v_{n-1} w_{n-1}$ and $w_{n-1} R_k$ are in $E(\Gamma)$, replace $R_k$ by $v_{n-1}$.
\item If only $v_{n-1} R_0$ and $w_{n-1} R_k$ are in $E(\Gamma)$, replace $R_0$ by $v_{n-1}$ and $R_k$ by $w_{n-1}$.
\item Suppose only $v_{n-1} w_{n-1}$ is in $E(\Gamma)$. If $\lambda$ is empty, then add $(v_{n-1}, w_{n-1})$. Otherwise, some $(x, y) \in \lambda$ is to be replaced by $(x, v_{n-1})$, $(y, w_{n-1})$. Any loose end that is connected to $v_{n-1}$ through the edges on the right hand side other than $v_{n-1} w_{n-1}$ must take the role of $x$, and any loose end that is connected to $w_{n-1}$ through the edges on the right hand side other than $v_{n-1} w_{n-1}$ must take the role of $y$. If one of $x$ and $y$ is determined in this way, the other one is determined by $\lambda$. If neither of them is determined in this way, any $(x, y) \in \lambda$ for which $x$ and $y$ are not connected through the outer edges on one side will do.
\end{itemize}

\section{Linear recurrence relations and characteristic polynomials}
If $k=3$, then simply by inspecting Table~\ref{intersections}, we find that the initial value of (\ref{sequence}) for $n=k$ is 1 when $(n_L, n_R, \lambda)$ is one of $(1, 1, \{(R_0, R_3), (R_1, R_2)\})$, $(2, 2, \{(L_3, R_0), (R_1, R_2)\})$, $(3, 3, \{(L_2, R_1), (R_0, R_3)\})$ and so on, and 0 in all other cases. Similarly, the initial values follow directly from the corresponding set $\{S_i\}$ for other values of $k$. For $n > k$, we have
\begin{equation}\label{sigma_relation}
u^{(n_L, n_R, \lambda)}_n = \sum_{((n_L, n'_R, \lambda'), (n_L, n_R, \lambda)) \in E(\sigma_k)} u^{(n_L, n_R', \lambda')}_{n-1}
\end{equation}
from which we we would like to obtain a linear recurrence relation $$u^{(n_L, n_R, \lambda)}_n + p_{d-1} u^{(n_L, n_R, \lambda)}_{n-1} + \dotsc + p_0 u^{(n_L, n_R, \lambda)}_{n-d}=0$$ that holds for all $n \geq n_0$ for some $n_0$. Such a recurrence relation is associated with a characteristic polynomial $$p(x) = p_dx^d + p_{d-1}x^{d-1}d + \dotsc + p_1x + p_0$$ with $p_d=1$. We say that $p(x)$ is a characteristic polynomial for $\{u_n\}$ for $n \geq n_0$ if $$p_d u_n + p_{d-1} u_{n-1} + \dotsc + p_0 u_{n-d}=0$$ holds for all $n \geq n_0$.

\begin{lemma}\label{product}
Let $\{u_n\}_{n \geq 1}$ be a sequence. The sum of two or more characteristic polynomials for $\{u_n\}$ is also a characteristic polynomial for $\{u_n\}$, and any product of a characteristic polynomial for $\{u_n\}$ by any polynomial is also a characteristic polynomial for $\{u_n\}$, for sufficiently large $n$. {\hfill \ensuremath{\Box}}
\end{lemma}

\begin{corollary}\label{gcd}
Let $\{u_n\}_{n \geq 1}$ be a sequence. The greatest common divisor of two or more characteristic polynomials for $\{u_n\}$ is also a characteristic polynomial for $\{u_n\}$, for sufficiently large $n$. {\hfill \ensuremath{\Box}}
\end{corollary}

Generally, for any directed graph $D$, we can define sequences $\{u^v_n\}_{v \in E(D)}$ by a set of initial conditions and the relation
\begin{equation}\label{general_relation}
u^v_n = \sum_{(v', v) \in E(D)} u^{v'}_{n-1}
\end{equation} corresponding to (\ref{sigma_relation}).

The monic polynomial of smallest degree that is a characteristic polynomial for all these sequences for a particular set $I$ of initial conditions and sufficiently large $n$ is called the \textit{minimal} characteristic polynomial on $(D, I)$. A polynomial that is a characteristic polynomial for all these sequences for \textit{any} set of initial conditions and sufficiently large $n$ is called a \textit{universal} characteristic polynomial on $D$, and the monic one with the smallest degree is called the \textit{minimal universal} characteristic polynomial on $D$. By Corollary~\ref{gcd}, the minimal characteristic polynomial on $(D, I)$ divides the minimal universal characteristic polynomial on $D$ regardless of $I$. Let $I_k$ denote the set of initial conditions for the sequences (\ref{sequence}).

In order to prove that a linear recurrence relation is valid for all of the sequences for all sufficiently large $n$, it suffices to verify that it holds for all the sequences for one value of $n$.

\begin{lemma}\label{recurrence}
Suppose $u^{(i)}_n = f^{(i)}(u^{(1)}_{n-1}, u^{(2)}_{n-1}, \dotsc, u^{(N)}_{n-1})$ for $n \geq n_0$, $1 \leq i \leq N$ where each $f^{(i)}$ is a linear function in $N$ variables, and that $g$ is a linear function such that
\begin{equation}
g(u^{(i)}_n, u^{(i)}_{n-1}, \dotsc, u^{(i)}_{n-d}) = 0
\end{equation}
for some $n=n_1 \geq n_0+d$ and all $i$, $1 \leq i \leq N$. Then (5) also holds for all $n \geq n_1$ and all $i$, $1 \leq i \leq N$.
\end{lemma}

\begin{proof}
By induction on $n$. If (5) holds for all $i$, then
\begin{align*}
g(u^{(i)}_{n+1}, u^{(i)}_n, \dotsc, u^{(i)}_{n-d+1}) = & g(f^{(i)}(u^{(1)}_n, u^{(2)}_n, \dotsc, u^{(N)}_n), \\
& \hspace{0.35cm} f^{(i)}(u^{(1)}_{n-1}, u^{(2)}_{n-1}, \dotsc, u^{(N)}_{n-1}), \dotsc, \\
& \hspace{0.35cm} f^{(i)}(u^{(1)}_{n-d}, u^{(2)}_{n-d}, \dotsc, u^{(N)}_{n-d})) \\
= & f^{(i)}(g(u^{(1)}_n, u^{(1)}_{n-1}, \dotsc, u^{(1)}_{n-d}), \\
& \hspace{0.7cm} g(u^{(2)}_n, u^{(2)}_{n-1}, \dotsc, u^{(2)}_{n-d}), \dotsc, \\
& \hspace{0.7cm} g(u^{(N)}_n, u^{(N)}_{n-1}, \dotsc, u^{(N)}_{n-d})) \\
= & f^{(i)}(0, 0, \dotsc, 0) \\
= & 0
\end{align*}
for all $i$.
\end{proof}

A directed graph $D$ can be partitioned into \textit{strongly connected components} (SCCs), i.e., maximal subgraphs in which any vertex can reach any other. An SCC $\xi \subseteq D$ for which there is no arc from a vertex in $D \backslash \xi$ to a vertex in $\xi$ is called a source SCC, and one for which there is no arc from a vertex in $\xi$ to a vertex in $D \backslash \xi$ is called a sink SCC.

Suppose that for each SCC $\xi \subseteq D$, the minimal universal characteristic polynomial on $\xi$ is known, called $p^{(\xi)}(x)$ of degree $d_{\xi}$. Then Algorithm 1 produces a characteristic polynomial on $(D, I)$. In view of Lemma~\ref{product}, the step in the bottom of the flowchart does not invalidate the universal characteristic polynomials on the remaining SCCs.
\newpage
\begin{mdframed}\textbf{Algorithm 1}

For one vertex $v$ from each sink SCC of $D$, set $P^{(v)}(x) = 1$, let $\rho(v) \subseteq D$ be induced by all vertices that can reach $v$, define the sequences $\{u^{(w)}_n\}_{w \in \rho(v)}$ by $I$ and (\ref{general_relation}), and execute this flowchart. Return the least common multiple of the polynomials it produces.
\newline

\begin{center}
\begin{tikzpicture}[node distance=2cm]
\node (start) [startstop] {Start};
\node (fin) [startstop, right of=start, xshift=3cm] {Finish};
\node (select) [process, below of=start, yshift=-1.5cm] {Select some source SCC $\xi \subseteq \rho(v)$};
\node (anymore) [decision, right of=select, xshift=3cm] {Any more SCCs?};
\node (seq0) [decisionbig, below of=select, yshift=-2.5cm] {Is $u^{(w)}_n=0$ for sufficiently high $n$ for all $w \in \xi$?};
\node (disregard) [process, right of=seq0, xshift=3cm] {Delete $\xi$ from $\rho(v)$};
\node (update) [processlong, below of=seq0, xshift=2.5cm, yshift=-2.5cm] {Set $P^{(v)}(x) \leftarrow p^{(\xi)}(x) P^{(v)}(x)$. For each $w \in \rho(v)$, replace each value $u^{(w)}_n$ by $p^{(\xi)}_{d_{\xi}} u^{(w)}_{n+d_{\xi}} + \dotsc + p^{(\xi)}_0 u^{(w)}_n$.};
\draw [arrow] (start) -- (select);
\draw [arrow] (select) -- (seq0);
\draw [arrow] (seq0) -- node[anchor=south]{Yes} (disregard);
\draw [arrow] (disregard) -- (anymore);
\draw [arrow] (anymore) -- node[anchor=south]{Yes} (select);
\draw [arrow] (anymore) -- node[anchor=east]{No} (fin);
\draw [thick,->,>=stealth] (0,-10.1) -- node[anchor=east]{No} (0,-11.5);
\draw [thick,->,>=stealth] (5,-11.5) -- (5,-9);
\end{tikzpicture}
\end{center}
\end{mdframed}

Analysis of $\sigma_k$ has shown that for $k=3$, each SCC is isomorphic to one out of 7 directed graphs with at most 24 vertices, and for $k=4$, each SCC is isomorphic to one out of 11 directed graphs with at most 122 vertices. These are small enough that we can find universal characteristic polynomials using iterated changes of variables. In order to find the minimal one for an SCC $\xi\subseteq D$, we can run the iterated changes of variables procedure several times and take the greatest common divisor, and/or we can check if some divisor of any given universal characteristic polynomial is also a characteristic polynomial if we set 
\[
u^{(v)}_k=
\begin{cases}
    1 &\quad\text{if }v = v_0\\
    0 &\quad\text{otherwise,}
\end{cases}
\]
regardless of $v_0 \in V(\xi)$. We set the initial values for $n=k$, as (\ref{sequence}) is only defined for $n \geq k$.

\section{Examples}
As an example of the analysis for a single SCC, consider $\xi \subset \sigma_3$ consisting of
\begin{align*}
\{ & (4, 32, \{(L_2, R_2), (L_3, R_1)\}), (4, 30, \{(L_2, R_1), (L_3, R_0)\}), \\
& (4, 26, \{(L_2, R_3), (L_3, R_0)\}), (4, 11, \{(L_2, R_2), (L_3, R_3)\}), \\
& (4, 27, \{(L_2, R_1), (L_3, R_2), (R_0, R_3)\}), (4, 16, \{(L_2, R_2), (L_3, R_1)\}), \\
& (4, 4, \{(L_2, R_1), (L_3, R_0)\}), (4, 23, \{(L_2, R_3), (L_3, R_1), (R_0, R_2)\}), \\
& (4, 10, \{(L_2, R_2), (L_3, R_1)\}), (4, 13, \{(L_2, R_2), (L_3, R_3), (R_0, R_1)\}), \\
& (4, 33, \{(L_2, R_1), (L_3, R_2), (R_0, R_3)\}), (4, 33, \{(L_2, R_3), (L_3, R_1), (R_0, R_2)\}), \\
& (4, 33, \{(L_2, R_2), (L_3, R_3), (R_0, R_1)\})\}
\end{align*}
and the corresponding sequences $\{u^{(1)}_n\}$ through $\{u^{(13)}_n\}$. The relations for these sequences, based on (\ref{sigma_relation}), are as follows for $n \geq 4$.
\newline

\begin{tabular}{ll}
$u_n^{(1)} = u_{n-1}^{(11)} + u_{n-1}^{(12)}$ & $u_n^{(8)} = u_{n-1}^{(5)}$ \\
$u_n^{(2)} = u_{n-1}^{(1)} + u_{n-1}^{(9)}$ & $u_n^{(9)} = u_{n-1}^{(8)}$ \\
$u_n^{(3)} = u_{n-1}^{(2)} + u_{n-1}^{(7)}$ & $u_n^{(10)} = u_{n-1}^{(8)}$ \\
$u_n^{(4)} = u_{n-1}^{(3)}$ & $u_n^{(11)} = u_{n-1}^{(10)} + u_{n-1}^{(13)}$ \\
$u_n^{(5)} = u_{n-1}^{(4)}$ & $u_n^{(12)} = u_{n-1}^{(11)}$ \\
$u_n^{(6)} = u_{n-1}^{(5)}$ & $u_n^{(13)} = u_{n-1}^{(12)}$ \\
$u_n^{(7)} = u_{n-1}^{(6)}$ & \\
\end{tabular}
\newline
\newline
Using only $\{u_n^{(1)}\}$, $\{u_n^{(2)}\}$, $\{u_n^{(3)}\}$ and $\{u_n^{(11)}\}$ (those with more than one term in their relation), we have $u_n^{(4)} = u_{n-1}^{(3)}$, $u_n^{(5)} = u_{n-2}^{(3)}$, $u_n^{(6)} = u_{n-3}^{(3)}$, $u_n^{(7)} = u_{n-4}^{(3)}$, $u_n^{(8)} = u_{n-3}^{(3)}$, $u_n^{(9)} = u_{n-4}^{(3)}$, $u_n^{(10)} = u_{n-4}^{(3)}$, $u_n^{(12)} = u_{n-1}^{(11)}$, $u_n^{(13)} = u_{n-2}^{(11)}$, and the relations
\begin{align}
u_n^{(1)} = u_{n-1}^{(11)} + u_{n-2}^{(11)}\\
u_n^{(2)} = u_{n-1}^{(1)} + u_{n-5}^{(3)}\\
u_n^{(3)} = u_{n-1}^{(2)} + u_{n-5}^{(3)}\\
u_n^{(11)} = u_{n-5}^{(3)} + u_{n-3}^{(11)}.
\end{align}
(6) and (9) gives
\begin{align}
u_n^{(1)} & = u_{n-6}^{(3)} + u_{n-4}^{(11)} + u_{n-7}^{(3)} + u_{n-5}^{(11)} \nonumber \\
& = u_{n-3}^{(1)} + u_{n-6}^{(3)} + u_{n-7}^{(3)}.
\end{align}
(8) and (10) gives
\begin{align}
u_n^{(1)} & = u_{n-3}^{(1)} + u_{n-7}^{(2)} + u_{n-11}^{(3)} + u_{n-8}^{(2)} + u_{n-12}^{(3)} \nonumber \\
& = u_{n-3}^{(1)} + u_{n-7}^{(2)} + u_{n-8}^{(2)} + u_{n-5}^{(1)} - u_{n-8}^{(1)} \nonumber \\
& = u_{n-3}^{(1)} + u_{n-5}^{(1)} - u_{n-8}^{(1)} + u_{n-7}^{(2)} + u_{n-8}^{(2)}
\end{align}
and (7) and (8) gives
\begin{align}
u_n^{(2)} & = u_{n-1}^{(1)} + u_{n-6}^{(2)} + u_{n-10}^{(3)} \nonumber \\
& = u_{n-1}^{(1)} + u_{n-6}^{(2)} + u_{n-5}^{(2)} - u_{n-6}^{(1)} \nonumber \\
& = u_{n-1}^{(1)} - u_{n-6}^{(1)} + u_{n-5}^{(2)} + u_{n-6}^{(2)}.
\end{align}
Finally, (11) and (12) gives
\begin{align*}
u_n^{(1)} = &  u_{n-3}^{(1)} + u_{n-5}^{(1)} - u_{n-8}^{(1)} + u_{n-8}^{(1)} - u_{n-13}^{(1)} + u_{n-12}^{(2)} \\
& + u_{n-13}^{(2)} + u_{n-9}^{(2)} - u_{n-14}^{(1)} + u_{n-13}^{(1)} + u_{n-14}^{(2)}\\
= & u_{n-3}^{(1)} + u_{n-5}^{(1)} + u_{n-9}^{(1)} - u_{n-13}^{(1)} - u_{n-1}^{14} + u_{n-5}^{(1)} - u_{n-8}^{(1)} \\
& -u_{n-10}^{(1)} + u_{n-13}^{(1)} + u_{n-6}^{(1)} - u_{n-9}^{(1)} - u_{n-11}^{(1)} + u_{n-14}^{(1)} \\
= & u_{n-3}^{(1)} + 2u_{n-5}^{(1)} + u_{n-6}^{(1)} - u_{n-8}^{(1)} - u_{n-10}^{(1)} - u_{n-11}^{(1)}
\end{align*}
which gives us a universal characteristic polynomial $x^{11}-x^8-2x^6-x^5+x^3+x+1$. However, it factorizes as $$x^{11}-x^8-2x^6-x^5+x^3+x+1 = (x-1)(x^4+x^3+x^2+x+1)(x^6-x^3-x-1),$$ and we can check whether some proper divisor is also a universal characteristic polynomial. Specifically, we can let $j \in \{1, \dotsc, 13\}$ and set $u^{(j)}_3=1$ and $u^{(i)}_3=0$ for $i \neq j$. It turns out that $u^{(i)}_{13} - u^{(i)}_{10}-u^{(i)}_8-u^{(i)}_7=0$ for all $i \in \{1, \dotsc, 13\}$ regardless of $j$. By Lemma~\ref{recurrence}, we then have $u^{(i)}_{n} - u^{(i)}_{n-3}-u^{(i)}_{n-5}-u^{(i)}_{n-6}=0$ for all $n \geq 13$ for all $i$, and since any set of initial values is a linear combination of the ones we have tried out, $x^6-x^3-x-1$ is a universal characteristic polynomial on $(\xi, I_3)$. Since it is irreducible and $\xi$ is nontrivial, it must be minimal.

As an example of an execution of the flowchart in Algorithm 1, let $v=(4, 16, \{(L_2, L_3), (R_1, R_2)\})$. The subgraph $\rho(v) \subset \sigma_3$ consists of the SCCs

\begin{align*}
\xi_1 = \{ & (4, 33, \{(L_2, R_0), (L_3, R_2), (R_1, R_3)\}), (4, 33, \{(L_2, R_0), (L_3, R_1), (R_2, R_3)\}), \\
& (4, 33, \{(L_2, R_0), (L_3, R_3), (R_1, R_2)\})\} \\
\xi_2 = \{ & (4, 33, \{(L_2, R_1), (L_3, R_0), (R_2, R_3)\}), (4, 33, \{(L_2, R_3), (L_3, R_0), (R_1, R_2)\}), \\
& (4, 33, \{(L_2, R_2), (L_3, R_0), (R_1, R_3)\})\} \\
\xi_3 = \{ & (4, 4, \{(L_2, R_0), (L_3, R_1)\}), (4, 26, \{(L_2, R_0), (L_3, R_3)\}), \\
& (4, 11, \{(L_2, R_3), (L_3, R_2)\}), (4, 27, \{(L_2, R_2), (L_3, R_1), (R_0, R_3)\}), \\
& (4, 16, \{(L_2, R_1), (L_3, R_2)\}), (4, 23, \{(L_2, R_1), (L_3, R_3), (R_0, R_2)\}), \\
& (4, 10, \{(L_2, R_1), (L_3, R_2)\}), (4, 30, \{(L_2, R_0), (L_3, R_1)\}), \\
& (4, 13, \{(L_2, R_3), (L_3, R_2), (R_0, R_1)\}), (4, 33, \{(L_2, R_2), (L_3, R_1), (R_0, R_3)\}), \\
& (4, 33, \{(L_2, R_1), (L_3, R_3), (R_0, R_2)\}), (4, 33, \{(L_2, R_3), (L_3, R_2), (R_0, R_1)\}), \\
& (4, 32, \{(L_2, R_1), (L_3, R_2)\})\} \\
\xi_4 = \{ & (4, 32, \{(L_2, R_2), (L_3, R_1)\}), (4, 30, \{(L_2, R_1), (L_3, R_0)\}), \\
& (4, 26, \{(L_2, R_3), (L_3, R_0)\}), (4, 11, \{(L_2, R_2), (L_3, R_3)\}), \\
& (4, 27, \{(L_2, R_1), (L_3, R_2), (R_0, R_3)\}), (4, 16, \{(L_2, R_2), (L_3, R_1)\}), \\
& (4, 4, \{(L_2, R_1), (L_3, R_0)\}), (4, 23, \{(L_2, R_3), (L_3, R_1), (R_0, R_2)\}), \\
& (4, 10, \{(L_2, R_2), (L_3, R_1)\}), (4, 13, \{(L_2, R_2), (L_3, R_3), (R_0, R_1)\}), \\
& (4, 33, \{(L_2, R_1), (L_3, R_2), (R_0, R_3)\}), (4, 33, \{(L_2, R_3), (L_3, R_1), (R_0, R_2)\}), \\
& (4, 33, \{(L_2, R_2), (L_3, R_3), (R_0, R_1)\})\} \\
\xi_5 = \{ & (4, 22, \{(L_2, L_3)\})\} \\
\xi_6 = \{ & (4, 7, \{(L_2, L_3), (R_0, R_3)\})\} \\
\xi_7 = \{ & (4, 19, \{(L_2, L_3), (R_2, R_3)\})\} \\
\xi_8 = \{ & (4, 1, \{(L_2, L_3), (R_0, R_3), (R_1, R_2)\})\} \\
\xi_9 = \{ & (4, 16, \{(L_2, L_3), (R_1, R_2)\}), (4, 4, \{(L_2, L_3), (R_0, R_1)\}), \\
& (4, 26, \{(L_2, L_3), (R_0, R_3)\}), (4, 11, \{(L_2, L_3), (R_2, R_3)\}), \\
& (4, 27, \{(L_2, L_3), (R_0, R_3), (R_1, R_2)\}), (4, 23, \{(L_2, L_3), (R_0, R_2), (R_1, R_3)\}), \\
& (4, 10, \{(L_2, L_3), (R_1, R_2)\}), (4, 30, \{(L_2, L_3), (R_0, R_1)\}), \\
& (4, 13, \{(L_2, L_3), (R_0, R_1), (R_2, R_3)\}), (4, 33, \{(L_2, L_3), (R_0, R_3), (R_1, R_2)\}), \\
& (4, 33, \{(L_2, L_3), (R_0, R_2), (R_1, R_3)\}), (4, 33, \{(L_2, L_3), (R_0, R_1), (R_2, R_3)\}), \\
& (4, 32, \{(L_2, L_3), (R_1, R_2)\})\}
\end{align*}
where $\xi_4$ is the SCC we analyzed earlier, and $v \in \xi_9$. The \textit{condensantion} of $\rho(v)$ (the resulting graph if each SCC is contracted to a single vertex) is shown in Figure~\ref{condensation}.
\newpage
The initial values in the source SCCs $\xi_1$ and $\xi_2$ are all zero, and they are deleted accordingly. $\xi_3$ then becomes a source SCC, and since all the initial values are zero here as well, it is also deleted, and we are left with $$\xi_4 \longrightarrow \xi_5 \longrightarrow \xi_6 \longrightarrow \xi_7 \longrightarrow \xi_8 \longrightarrow \xi_9.$$ We have just seen that the minimal universal characteristic polynomial on $\xi_4$ is $x^6-x^3-x-1$. $\xi_9$ is isomorphic to $\xi_4$, and therefore has the same minimal universal characteristic polynomial. The remaining SCCs are trivial, and this gives $P^{(v)}(x)=(x^6-x^3-x-1)^2$.

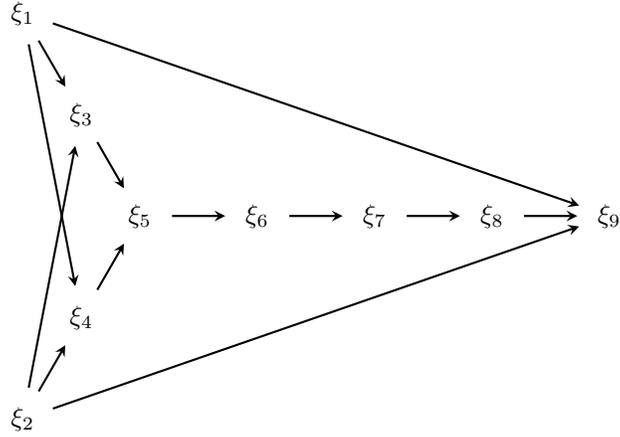
\begin{figure}
\begin{center}
\begin{tikzpicture}[node distance=1cm]
\node (xi1) [SCC] {$\xi_1$};
\node (xi3) [SCC, xshift=0.78cm, yshift=-1.351cm] {$\xi_3$};
\node (xi5) [SCC, xshift=1.56cm, yshift=-2.702cm] {$\xi_5$};
\node (xi4) [SCC, xshift=0.78cm, yshift=-4.053cm] {$\xi_4$};
\node (xi2) [SCC, xshift=0cm, yshift=-5.404cm] {$\xi_2$};
\node (xi6) [SCC, xshift=3.12cm, yshift=-2.702cm] {$\xi_6$};
\node (xi7) [SCC, xshift=4.68cm, yshift=-2.702cm] {$\xi_7$};
\node (xi8) [SCC, xshift=6.24cm, yshift=-2.702cm] {$\xi_8$};
\node (xi9) [SCC, xshift=7.8cm, yshift=-2.702cm] {$\xi_9$};
\draw [arrow] (xi1) -- (xi3);
\draw [arrow] (xi1) -- (xi4);
\draw [arrow] (xi1) -- (xi9);
\draw [arrow] (xi2) -- (xi3);
\draw [arrow] (xi2) -- (xi4);
\draw [arrow] (xi2) -- (xi9);
\draw [arrow] (xi3) -- (xi5);
\draw [arrow] (xi4) -- (xi5);
\draw [arrow] (xi5) -- (xi6);
\draw [arrow] (xi6) -- (xi7);
\draw [arrow] (xi7) -- (xi8);
\draw [arrow] (xi8) -- (xi9);
\end{tikzpicture}
\caption{Condensation of $\rho((4, 16, \{(L_2, L_3), (R_1, R_2)\}))$}
\label{condensation}
\end{center}
\end{figure}

Generally, checking divisors ought to be done after executing the flowchart like we did in the example with a single SCC.

\section{Results}
For $k=3$, let the polynomials $P^{(1)} \dotsc, P^{(8)}$ be given by
\begin{align*}
P^{(1)}(x) & = x-1 \\
P^{(2)}(x) & = x+1 \\
P^{(3)}(x) & = x^2+1 \\
P^{(4)}(x) & = x^2-x+1 \\
P^{(5)}(x) & = x^2+x+1 \\
P^{(6)}(x) & = x^4-x^2+1 \\
P^{(7)}(x) & = x^6-x^3-x-1 \\
P^{(8)}(x) & = x^{14}-x^{12}-x^8+x^6-x^4-2x^2-1
\end{align*}
Table~\ref{charsigma3} gives the minimal universal characteristic polynomials for each class of SCCs under graph isomorphism.

\begin{table}
\begin{center}
\begin{tabular}{|c|c|}\hline
SCC size & Characteristic polynomial \\ \hline
1 & 1 \\
2 & $P^{(1)}(x) P^{(2)}(x)$ \\
3 & $P^{(1)}(x) P^{(5)}(x)$ \\
4 & $P^{(1)}(x) P^{(2)}(x) P^{(3)}(x)$ \\
12 & $P^{(1)}(x) P^{(2)}(x) P^{(3)}(x) P^{(4)}(x) P^{(5)}(x) P^{(6)}(x)$ \\
13 & $P^{(7)}(x)$ \\
24 & $P^{(8)}(x)$ \\ \hline
\end{tabular}
\caption{Minimal universal characteristic polynomials for SCCs in $\sigma_3$}
\label{charsigma3}
\end{center}
\end{table}

Algorithm 1 yields $$P(x)=P_1(x) P_2(x) P_3(x) P_4(x) P_5(x) P_6(x) \left(P_7(x)\right)^2 P_8(x)$$ of degree 38 as a characteristic polynomial on $(\sigma_3, I_3)$, which we have verified to be minimal by checking all maximal divisors. The corresponding recurrence relation $$P_{38}u_n + P_{37}u_{n-1} + \dotsc + P_1 u_{n-37} + P_0 u_{n-38}=0$$ holds for each signature for $n=45$, and thus, by Lemma~\ref{recurrence}, for all $n \geq 45$. It follows that it holds for $u_n=h_3(n)$ for $n \geq 45$. However, we have $h_3(1) = 1$ and $h_3(2)=4$ and can verify that the recurrence relation also holds for $39 \leq n < 45$, from which we can conclude that it holds for all $n \geq 39$. 

Similarly for $k=4$, let
\begin{align*}
Q^{(1)}(x) = & x-1 \\
Q^{(2)}(x) = & x+1 \\
Q^{(3)}(x) = & x^2+1 \\
Q^{(4)}(x) = & x^4+1 \\
Q^{(5)}(x) = & x^4-x^3+x^2-x+1 \\
Q^{(6)}(x) = & x^4+x^3+x^2+x+1 \\
Q^{(7)}(x) = & x^8-x^6+x^4-x^2+1 \\
Q^{(8)}(x) = & x^8-x^4-x^2+x-1 \\
Q^{(9)}(x) = & x^8-x^4-x^2-x-1 \\
Q^{(10)}(x) = & x^{10}+x^8-2x^6-x^4+x^2-1 \\
Q^{(11)}(x) = & x^{21}-x^{20}+x^{19}-2x^{18}+2x^{17}-3x^{16}+3x^{15}-3x^{14}+3x^{13}-3x^{12} \\
& +2x^{11}-2x^{10}+x^9-2x^8+2x^7-x^6+x^5-x^4+x^3-x^2+x-1 \\
Q^{(12)}(x) = & x^{22}-2x^{18}-x^{15}-2x^{14}-x^{12}+x^{11}+3x^{10} \\
& -x^9-x^8+2x^6+x^5+x^4-x^3+x+1 \\
Q^{(13)}(x) = & x^{37}+x^{34}-x^{29}+x^{27}-x^{25}+x^{20}+3x^{17} \\
& +3x^{14}+8x^{12}+x^9+3x^7-3x^5-1
\end{align*}
Table~\ref{charsigma4} gives the minimal universal characteristic polynomials for each class of SCCs.

\begin{table}
\begin{center}
\begin{tabular}{|c|c|}\hline
SCC size & Characteristic polynomial \\ \hline
1A & 1 \\
1B & $Q^{(1)}(x)$ \\
2 & $Q^{(1)}(x) Q^{(2)}(x)$ \\
4 & $Q^{(1)}(x) Q^{(2)}(x) Q^{(3)}(x)$ \\
5 & $Q^{(1)}(x) Q^{(6)}(x)$ \\
10 & $Q^{(1)}(x) Q^{(2)}(x) Q^{(5)}(x) Q^{(6)}(x)$ \\
20 & $Q^{(1)}(x) Q^{(2)}(x) Q^{(3)}(x) Q^{(5)}(x) Q^{(6)}(x) Q^{(7)}(x)$ \\
46 & $Q^{(2)}(x) Q^{(11)}(x)$ \\
66 & $Q^{(1)}(x) Q^{(2)}(x) Q^{(3)}(x) Q^{(12)}(x)$ \\
114 & $Q^{(1)}(x) Q^{(2)}(x) Q^{(4)}(x) Q^{(8)}(x) Q^{(9)}(x) Q^{(10)}(x)$ \\
122 & $Q^{(2)}(x) Q^{(3)}(x) Q^{(5)}(x) Q^{(7)}(x) Q^{(11)}(x) Q^{(13)}(x)$ \\ \hline
\end{tabular}
\caption{Minimal universal characteristic polynomials for SCCs in $\sigma_4$}
\label{charsigma4}
\end{center}
\end{table}

Again using Algorithm 1 and checking divisors, the minimal characteristic polynomial on $(\sigma_4, I_4)$ was found to be
\begin{align}
& \left(Q^{(1)}(x) Q^{(2)}(x) Q^{(3)}(x) Q^{(5)}(x) Q^{(6)}(x) Q^{(7)}(x) Q^{(11)}(x)\right)^2 \nonumber \\
& Q^{(4)}(x) Q^{(8)}(x) Q^{(9)}(x) Q^{(10)}(x) Q^{(12)}(x) Q^{(13)}(x)
\end{align}
of degree 171.

Let 
\begin{align*}
Q(x) = & Q^{(1)}(x) \left(Q^{(2)}(x) Q^{(3)}(x) Q^{(5)}(x) Q^{(7)}(x) Q^{(11)}(x) \right)^2 \\
& Q^{(4)}(x) Q^{(8)}(x) Q^{(9)}(x) Q^{(10)}(x) Q^{(12)}(x) Q^{(13)}(x)
\end{align*}
of degree 162, such that (13) is equal to $$Q(x) Q^{(1)}(x) \left(Q^{(6)}(x)\right)^2.$$ There are only 25 signatures $(n_L, n_R, \lambda)$ for which $Q(x) Q^{(1)}(x) Q^{(6)}(x)$ is \textit{not} a characteristic polynomial for $\{u^{(n_L, n_R, \lambda)}_n\}$ for $n=176$, and none of them can reach any of the signatures corresponding to Hamiltonian cycles in $G(n, 4)$. Therefore, $Q(x) Q^{(1)}(x) Q^{(6)}(x)$ is a characteristic polynomial for $h_4(n)$ for $n \geq 176$ by Lemma~\ref{recurrence}, and since $Q^{(1)}(x) Q^{(6)}(x) = x^5-1$, it follows that $$Q_{162}h_4(n) + Q_{161}h_4(n-1) + \dotsc + Q_1 h_4(n-161) + Q_0 h_4(n-162)$$ is periodic with period 5, starting from $n=171$. But all the first five terms are zero, and thus, $Q(x)$ is a characteristic polynomial for $h_4(n)$ for $n \geq 171$. Moreover, with $h_4(1)=1$, $h_4(2)=0$ and $h_4(3)=3$, we can verify that the recurrence relation also holds for $163 \leq n < 171$ and conclude that it holds for all $n \geq 163$.

\begin{theorem}
With $P$ and $Q$ defined as above, we have $$P_{38} h_3(n) + P_{37} h_3(n-1) + \dotsc + P_1 h_3(n-37) + P_0 h_3(n-38) = 0 \text{ for } n \geq 39,$$ $$Q_{162} h_4(n) + Q_{161} h_4(n-1) + \dotsc + Q_1 h_4(n-161) + Q_0 h_4(n-162) = 0 \text{ for } n \geq 163.$$ {\hfill \ensuremath{\Box}}
\end{theorem}

Appendix A gives the initial values of $h_3(n)$, $h_4(n)$ and the coefficients of $P(x)$ and $Q(x)$.

It might be desirable to write $h_k(n)$ as a sum of the $n$th terms of several recursive sequences, each with a characteristic polynomial of relatively small degree. One possible first step is to distinguish between admissible subgraphs of $G'(n, k)$ with an even and an odd number of loose ends on each side. Partitioning the graph of signatures $\sigma_k$ into $\sigma^{(e)}_k$ and $\sigma^{(o)}_k$ accordingly, we can replace (\ref{hamiltonian}) by
\begin{equation*}
h^{(e)}_k(n) = \sum_{(n_L, n_R, \lambda) \in H_k \cap V \left( \sigma^{(e)}_3 \right)} u^{(n_L, n_R, \lambda)}_n
\end{equation*}
and
\begin{equation*}
h^{(o)}_k(n) = \sum_{(n_L, n_R, \lambda) \in H_k \cap V \left( \sigma^{(o)}_3 \right) } u^{(n_L, n_R, \lambda)}_n
\end{equation*}
such that $h_k(n) = h^{(e)}_k(n) + h^{(o)}_k(n)$. Furthermore, as touched upon when we considered characteristic polynomials for $h_4(n)$ above, a divisor of the characteristic polynomial of the form (or dividing) $x^j-1$ can be eliminated if we introduce a periodic function of period $j$ on the right hand side of the recurrence relation.

Here is a brief summary of how this plays out for $k=3$. $V \left( \sigma^{(e)}_3 \right)$ consists of signatures with $n_L, n_R \in \{1, 4, 6, 7, 10, \dotsc\}$ (confer Table~\ref{intersections}). With the same analysis as before, we obtain the initial values $1, 2, 0, 0, 0, 8, 7, 0, 9, 12, 11, 36$ of $h_k^{(e)}(n)$ for $1 \leq n \leq 12$, and the recurrence relation
\[
\begin{split}
h_k^{(e)}(n) - 2h_k^{(e)}(n-3) - 2h_k^{(e)}(n-5)\\- h_k^{(e)}(n-6) + 2h_k^{(e)}(n-8) + 2h_k^{(e)}(n-9)\\+ h_k^{(e)}(n-10) + 2h_k^{(e)}(n-11) + h_k^{(e)}(n-12)
\end{split}=
\begin{cases}
    8 &\quad\text{if }n \equiv 2 (\operatorname{mod} 4)\\
    0 &\quad\text{otherwise}
\end{cases}
\]
for $n \geq 13$. A periodic function satisfying the same recurrence relation is given by
\[
u^{(e)}_n=
\begin{cases}
    2 &\quad\text{if }n \equiv 2 (\operatorname{mod} 4)\\
    0 &\quad\text{otherwise.}
\end{cases}
\]
Similarly, $V \left( \sigma^{(o)}_3 \right)$ consists of signatures with $n_L, n_R \in \{2, 3, 5, 8, 9, \dotsc\}$, the initial values of $h_k^{(o)}(n)$ for $1 \leq n \leq 14$ are $0, 2, 0, 6, 0, 8, 0, 6, 0, 12, 0, 32, 0, 58$, and
\[
\begin{split}
h_k^{(o)}(n) - h_k^{(o)}(n-2) - h_k^{(o)}(n-6)\\+ h_k^{(o)}(n-8) - h_k^{(o)}(n-10)\\- 2h_k^{(o)}(n-12) - h_k^{(o)}(n-14)
\end{split}=
\begin{cases}
    16 &\quad\text{if }n \equiv 0 \text{ or } 2 (\operatorname{mod} 12)\\
    -16 &\quad\text{if }n \equiv 8 (\operatorname{mod} 12)\\
    0 &\quad\text{otherwise}
\end{cases}
\]
for $n \geq 15$. A periodic function satisfying the same recurrence relation is given by
\[
u^{(o)}_n=
\begin{cases}
    -12 &\quad\text{if }n \equiv 0 (\operatorname{mod} 4)\\
    4 &\quad\text{if }n \equiv 4 \text{ or }8 (\operatorname{mod}12)\\
    0 &\quad\text{otherwise.}
\end{cases}
\]
Thus, taking $h_3^{(\alpha)}(n) = h_k^{(e)}(n) - u^{(e)}_n$, $h_3^{(\beta)}(n) = h_k^{(o)}(n) - u^{(o)}_n$ and $h_3^{(\gamma)}(n) = u^{(e)}_n + u^{(o)}_n$ leads to the following alternative to Theorem 5.1 for $k=3$.

\begin{proposition}
$h_3(n) = h_3^{(\alpha)}(n) + h_3^{(\beta)}(n) + h_3^{(\gamma)}(n)$ for $n \geq 1$ where the initial values of $h_3^{(\alpha)}(n)$ are $1, 0, 0, 0, 0, 6, 7, 0, 9, 10, 11, 36$ and the characteristic polynomial is $(P^{(7)}(x))^2$, the initial values of $h_3^{(\beta)}(n)$ are $0, 2, 0, 2, 0, 8, 0, 2, 0, 12,$ $0, 44, 0, 58$ and the characteristic polynomial is $P^{(8)}(x)$, and $h_3^{(\gamma)}(n)$ is periodic with the terms $0, 2, 0, 4, 0, 2, 0, 4, 0, 2, 0, -12$ repeating. {\hfill \ensuremath{\Box}}
\end{proposition}

\appendix
\section*{Appendix A: Initial values for $h_3(n)$ and $h_4(n)$  and coefficients of their characteristic polynomials}
\begin{tabular}{|ccccc|}
\hline
$n$ & $h_3(n)$ & $P_n$ & $h_4(n)$ & $Q_n$ \\ \hline
0 & & 1 & & -1 \\
1 & 1 & 2 & 1 & 2 \\
2 & 4 & 3 & 0 & -2 \\
3 & 0 & 6 & 3 & 3 \\
4 & 6 & 5 & 0 & -10 \\
5 & 0 & 6 & 5 & 12 \\
6 & 16 & 3 & 6 & -13 \\
7 & 7 & -2 & 7 & 30 \\
8 & 6 & 0 & 0 & -51 \\
9 & 9 & -6 & 3 & 38 \\
10 & 24 & -2 & 30 & -56 \\
11 & 11 & -4 & 11 & 124 \\
12 & 68 & 4 & 6 & -125 \\
13 & 26 & 0 & 26 & 91 \\
14 & 88 & -2 & 56 & -205 \\
15 & 75 & -6 & 53 & 260 \\
16 & 150 & -3 & 80 & -87 \\
17 & 102 & -10 & 34 & 150 \\
18 & 316 & -7 & 216 & -395 \\
19 & 152 & 0 & 152 & 131 \\
20 & 436 & 2 & 240 & 165 \\
21 & 399 & 6 & 192 & 205 \\
22 & 664 & 2 & 462 & 7 \\
23 & 667 & 6 & 598 & -903 \\
24 & 1352 & -4 & 750 & 680 \\
25 & 975 & -2 & 580 & -370 \\
26 & 2344 & -2 & 1300 & 1934 \\
27 & 1998 & 0 & 1245 & -2628 \\
28 & 3618 & -2 & 2156 & 1773 \\
29 & 3683 & 4 & 2117 & -3209 \\
30 & 6440 & 4 & 3346 & 5354 \\
31 & 5766 & 2 & 4123 & -4076 \\
32 & 11350 & -2 & 6432 & 4494 \\
33 & 10230 & 0 & 5822 & -8135 \\
34 & 18160 & 0 & 9418 & 7988 \\
35 & 18865 & -2 & 10400 & -6156 \\
36 & 30542 & -1 & 18954 & 10028 \\
37 & 31339 & 0 & 18944 & -11486 \\
38 & 53736 & 1 & 26144 & 7586 \\
39 & & & 30826 & -9113 \\
40 & & & 51380 & 13095 \\ \hline
\end{tabular}
\quad
\begin{tabular}{|ccc|}
\hline
$n$ & $h_4(n)$ & $Q_n$ \\ \hline
41 & 53997 & -8205 \\
42 & 75718 & 5742 \\
43 & 86344 & -10580 \\
44 & 148170 & 7544 \\
45 & 158858 & -266 \\
46 & 212980 & 3517 \\
47 & 249429 & -3091 \\
48 & 404094 & -7967 \\
49 & 454818 & 10326 \\
50 & 619370 & -7766 \\
51 & 719698 & 18242 \\
52 & 1133158 & -25199 \\
53 & 1287158 & 19498 \\
54 & 1757538 & -23046 \\
55 & 2078521 & 31413 \\
56 & 3138352 & -24322 \\
57 & 3669321 & 19086 \\
58 & 5046406 & -27751 \\
59 & 5969030 & 26018 \\
60 & 8756276 & -18757 \\
61 & 10333827 & 26544 \\
62 & 14336756 & -32452 \\
63 & 17180271 & 22362 \\
64 & 24489280 & -21291 \\
65 & 29332438 & 24233 \\
66 & 40707530 & -9473 \\
67 & 49039712 & -6804 \\
68 & 68644708 & 10117 \\
69 & 82999229 & -23260 \\
70 & 115391926 & 40604 \\
71 & 140133552 & -39327 \\
72 & 192987894 & 33736 \\
73 & 235458361 & -36746 \\
74 & 326022466 & 27198 \\
75 & 398576453 & -10072 \\
76 & 543868768 & 7580 \\
77 & 668462960 & -7641 \\
78 & 921571358 & 1552 \\
79 & 1133200253 & -8165 \\
80 & 1534180524 & 24431 \\
81 & 1898214186 & -29303 \\ \hline
\end{tabular}

\begin{tabular}{|ccc|}
\hline
$n$ & $h_4(n)$ & $Q_n$ \\ \hline
82 & 2601335036 & 33015 \\
83 & 3217850655 & -40137 \\
84 & 4336332804 & 34426 \\
85 & 5393967609 & -16912 \\
86 & 7345848928 & 4473 \\
87 & 9131563039 & 11028 \\
88 & 12260475150 & -31833 \\
89 & 15327541552 & 42174 \\
90 & 20748259282 & -41404 \\
91 & 25918078745 & 39960 \\
92 & 34700887198 & -31311 \\
93 & 43555191769 & 13171 \\
94 & 58629617490 & 787 \\
95 & 73554425379 & -9286 \\
96 & 98250605966 & 17737 \\
97 & 123781143317 & -19638 \\
98 & 165799513320 & 11823 \\
99 & 208828178672 & -3909 \\
100 & 278309480680 & -3465 \\
101 & 351735275884 & 13053 \\
102 & 469167706296 & -18753 \\
103 & 593050414339 & 17089 \\
104 & 788699042066 & -13954 \\
105 & 999630887903 & 8952 \\
106 & 1328619760504 & -216 \\
107 & 1684786193306 & -7352 \\
108 & 2235846181590 & 11071 \\
109 & 2841083989352 & -13733 \\
110 & 3765130309158 & 14258 \\
111 & 4788053484820 & -10650 \\
112 & 6341062586720 & 6458 \\
113 & 8076169517439 & -2911 \\
114 & 10676699727520 & -1068 \\
115 & 13611490424942 & 4344 \\
116 & 17990921225536 & -5168 \\
117 & 22962336522505 & 5384 \\
118 & 30294726883974 & -5326 \\
119 & 38707246756884 & 4113 \\
120 & 51065984371156 & -2619 \\
121 & 65302052798697 & 1677 \\
122 & 86006615216140 & -532 \\ \hline
\end{tabular}
\quad
\begin{tabular}{|ccc|}
\hline
$n$ & $h_4(n)$ & $Q_n$ \\ \hline
123 & 110103773896354 & -854 \\
124 & 145010577267434 & 1600 \\
125 & 185761373321580 & -2266 \\
126 & 244298061980386 & 2975 \\
127 & 313278174180143 & -3287 \\
128 & 411958560266624 & 3051 \\
129 & 528566740244249 & -2690 \\
130 & 694237816258292 & 1946 \\
131 & 891605743330819 & -830 \\
132 & 1170830974647758 & -229 \\
133 & 1504404142946959 & 1128 \\
134 & 1973697694412892 & -1880 \\
135 & 2538183157473095 & 2277 \\
136 & 3328991719090480 & -2184 \\
137 & 4282979219051053 & 1866 \\
138 & 5613375501130086 & -1325 \\
139 & 7227354535798072 & 664 \\
140 & 9468956256988646 & -55 \\
141 & 12196622849583337 & -386 \\
142 & 15970732181345526 & 686 \\
143 & 20584346089684544 & -804 \\
144 & 26943410940175302 & 733 \\
145 & 34740880947606620 & -583 \\
146 & 45454097795928882 & 387 \\
147 & 58639873463485310 & -184 \\
148 & 76692548521685142 & 15 \\
149 & 98978960938944210 & 82 \\
150 & 129407167837177966 & -136 \\
151 & 167087245796774770 & 141 \\
152 & 218370361987361104 & -112 \\
153 & 282059107544821099 & 80 \\
154 & 368529000886276550 & -46 \\
155 & 476193350045990337 & 20 \\
156 & 621960955996700508 & 0 \\
157 & 803944918169270089 & -9 \\
158 & 1049795010065254566 & 12 \\
159 & 1357402794266151258 & -11 \\
160 & 1771950050554423404 & 6 \\
161 & 2291899313078463756 & -3 \\
162 & 2991220108882081740 & 1 \\
& & \\ \hline
\end{tabular}

\end{document}